\documentclass[a4paper,12pt,reqno]{amsart}
\usepackage{mathrsfs}
\usepackage{amsmath,amssymb,latexsym,amsfonts,amscd}
\usepackage{showlabels}
\usepackage{longtable}
\title{On the canonical volume of 3-folds of general type with $P_{12}\geq 2$}
\author{Lingzi Hou}

\address{\rm School of Mathematical Sciences, Fudan University,
Shanghai 200433, China}
\email{10300180105@fudan.edu.cn}

\thanks{The author was
supported by  National Talents Training Base for Basic Research and Teaching of Natural Science of China, No.J1103105}
\numberwithin{equation}{section}

\newcommand\Vol{\text{\rm Vol}}

\newcommand\tki{\tilde{\chi}}

\newtheorem{thm}{Theorem}[section]
\newtheorem{lem}[thm]{Lemma}
\newtheorem{cor}[thm]{Corollary}

\theoremstyle{definition}
\newtheorem{defn}[thm]{Definition}

\newtheorem{rem}[thm]{Remark}
\theoremstyle{remark}

\begin{document}
\begin{abstract}
Let $V$ be a nonsingular projective 3-fold of general type.  When the pluricanonical section index $\delta(V)>12$, Chen-Chen \cite{Chen3} has a complete list of the possibility for the weighted basket ${\mathbb B}(V)$. However the possibility of ${\mathbb B}(V)$ could be infinite in the situation $\delta(V)\leq 12$, which is the main challenge to the classification. In this paper we mainly study the case with $\delta(V)=12$ and show that $\Vol(V)\geq \frac{31}{48048}$, which improves the corresponding result of Chen--Chen \cite[Prop. 4.9(5)]{Chen3}. 
\end{abstract}

\maketitle
\pagestyle{myheadings} \markboth{\hfill Lingzi Hou
\hfill}{\hfill On the canonical volume of threefolds\hfill}




\section{\bf Introduction}
One of the fundamental problems in birational geometry is to study the distribution of discrete birational invariants of varieties in question.  Among those birational invariants, the canonical volume plays the key role. Given a nonsingular projective $n$-fold $Y$, the canonical volume 
$$\Vol(Y):= \text{lim sup}_{n\in {\mathbb Z}^+}\frac{n!\ h^0(Y, mK_Y)}{m^n}.$$ 

The remarkable theorem, proved by Hacon-McKernan \cite{H-M}, Takayama \cite{Tak} and Tsuji \cite{Tsuji}, says that there exists a constant $v_n>0$ so that $\Vol(Y)\geq v_n$ for
any $n\geq 3$. However the constant $v_n$ is not explicitly known except that one has  $v_3\geq \frac{1}{1680}$, which  was proved   by Chen-Chen \cite{Chen1, Chen2, Chen3}.  Restricting our interest to 3-folds,  we are absorbed by  those beautiful calculations of Chen--Chen. Thanks to the guidance of Meng Chen, I am able to consider one of the boundary case here that, when $P_{12}\geq 12$, the possibility for the weighted basket could be infinite and thus it is necessary to slightly extend the method developed by Chen--Chen to find the lower bound of $\Vol(Y)$.    

To be precise, we always consider a nonsingular projective 3-fold $V$ of general type. Pick up any minimal model $X$ of $V$. The so-called weighted basket $${\mathbb B}(V):={\mathbb B}(X)=\{B_X, \chi(\mathcal{O}_X), P_2(X)\}$$
where $B_X$ is the Reid basket (see \cite{YPG}).  Recall from Chen--Chen \cite{Chen3} that the pluricanonical section index 
$$\delta(V)=\delta(X):=\text{min}_{m\in {\mathbb Z}^+}\{m|P_m(X)\geq 2\}.$$
By Chen--Chen \cite{Chen1}, one has $1\leq \delta(V)\leq 18$.  So far, ${\mathbb B}(V)$ has been completely classified when $\delta(V)>12$ (see Chen-Chen \cite[Theorem 1.4]{Chen3}).  When $\delta(V)\leq 12$, though Chen--Chen has given effective lower bounds for $K_X^3$, it is interesting to see if those are optimal. It is with this motivation that we are going to study the case $\delta(V)=12$ in this paper. Our main result is the following:
 
\begin{thm}\label{HLZ}
Let $V$ be a nonsingular projective 3-fold of general
type with $P_{12}\geq 2$. Then \begin{itemize}
\item[(1)] When $P_{12}=2$, $\mathbb{B}(V)$ imust be one of the cases listed in 4.3 and 5.3;

\item[(2)] $Vol(V)\geq\frac{31}{48048}$ and the equality holds
 if and only if $\mathbb{B}(V)=\{B_H, 4,0\}$ where 
 {\small $$B_H=\{9\times(1,2),(7,16),(3,7),2\times (5,13),5\times (1,3),(2,7),(3,11),(1,4)\}.$$}
\end{itemize}
\end{thm}

Theorem \ref{HLZ} has slightly improved the lower bound $\frac{1}{1560}$ obtained by Chen--Chen \cite[Proposition 4.9(5)]{Chen3}.

We briefly explain the idea of this paper. Assume $P_{12}\geq 2$. Though $\{{\mathbb B}(V)|\delta(V)=12\}$ might be an infinite set, we may consider two sub-cases:(1) $P_{12}\geq 3$; (2) $P_{12}=2$.  The first case is better since Chen--Chen's inequality in \cite{Chen3}  can be applied to estimate the lower bound of $K_X^3$. In the second case, we first deduce that $P_2=0$ and $n^0_{1,6}=0,1$. By Chen--Chen's formulae in \cite{Chen1}, we can see that $\chi(\mathcal{O}_X)$ is bounded from above. Thus, theoretically, this is a computable problem. In fact, for the situation $n_{1,6}^0=0$, we may use Chen-Chen's formula to realize the calculation.  When $n_{1,6}^0=1$, known formula is not useful for us and we need to deduce the expression for $B^{(n)}$ instead. Fortunately, the result is favorable as there are very few outputs. So we could eventually conclude our main statement.

Essentially affected by Chen--Chen's work, we more or less keep the same notation and concepts as in \cite{Chen1, Chen2, Chen3}. 
\medskip

Hereby I would like to express my sincere gratitude to professor Meng Chen who guided me to this topic and helped me a lot in completing this paper. 

\section{\bf Preliminaries}

We recall the invariants and structure of baskets developed by Chen--Chen in \cite{Chen1}. 

\subsection{\bf The basket of terminal quotient singularities and packings}
\begin{defn}
By a 3-dimensional terminal quotient singularity $Q=\frac{1}{r}(1,-1,b)$, we mean one which is analytically isomorphic to the quotient of ($\mathbb{C}^3$, 0) by
a cyclic group action $\varepsilon$:
$\varepsilon(x,y,z)=(\varepsilon{x},\varepsilon^{-1}{y},\varepsilon^bz)$, where
 $r$ is a positive integer, $\varepsilon$ is a fixed $r$-th
 primitive root of 1, the integer $b$ is coprime to $r$ and $0<b<r$. By
 replacing $\varepsilon$ with another primitive root of 1 and changing
 the ordering of coordinates, we may and do assume that $0<b\leq\frac{r}{2}$.
\end{defn}

\begin{defn}
{\it A basket} $B$ (of terminal orbifolds) is a collection (allowing weights) of
terminal quotient singularities of type $\frac{1}{r_i}(1,-1,b_i)$, $i\in{I}$ where $I$
is a finite index set. For simplicity, we will always denote a terminal quotient
singularity $\frac{1}{r}(1,-1,b)$ simply by the pair $(b,r)$ when no confusion is likely. So we will write a basket as: $B=\{m_i\times(b_i,r_i)|i=1,2,\ldots,t\}=\{(b_j,r_j)|j\in{J}\}, m_i\in\mathbb{Z^{+}}$, where each $m_i$ denotes
the weight. A generalized basket means a collection of pairs of integers $(b,r)$ with $0<b<r$, but not necessary requiring $b$, $r$ to be coprime each other.
\end{defn}

\begin{defn}
Given a generalized basket $(b, r)$ with $b\leq\frac{r}{2}$ and a fixed integer $n>0$.
Let $\delta:=\lfloor\frac{bn}{r}\rfloor$.
Then $\frac{\delta+1}{n}>\frac{b}{r}\geq\frac{\delta}{n}$.
We define
$\Delta^n(b,r):=\delta{bn}-\frac{(\delta^2+\delta)}{2}r,
~~\sigma(B):=\underset{i\in{I}}{\sum}b_i~,~\sigma'(B):=
\underset{i\in{I}}{\sum}\frac{b_i^2}{r_i}$. 
\end{defn}

One can see that $\Delta^n(b,r)$ is a non-negative integer.
For a generalized basket $B=\{(b_i,r_i)|i\in{I}\}$ and a
fixed $n>0$, we define $\Delta^n(B)=\underset{i\in{I}}{\sum}\Delta^n(B_i)$ where $B_i=\{(b_i,r_i)\}$.


\begin{defn}
Given a basket $$B=\{(b_1,r_1), (b_2,r_2), (b_3,r_3), \ldots, (b_t,r_t)\}, $$we
call the basket $B'=\{(b_1+b_2, r_1+r_2), (b_3, r_3), \ldots, (b_t,r_t)\}$ {\it a packing}
of $B$ (and $B$ is an unpacking of $B'$ ), written as $B\succeq{B'}$. If, furthermore, $b_1r_2-b_2r_1=1$, we call $B\succ{B'}$ {\it a prime packing}. \end{defn}

\begin{lem} (Chen--Chen \cite[Lemma 2.8]{Chen1}) 
Let $B\succeq{B'}$ be any packing between generalized baskets. Keep the same notation as above. Then
\begin{itemize}
\item[(1)]$\Delta^n(B)\geq\Delta^n(B')$ for all $n\geq2$;
\item[(2)]$\Delta^n(B)=\Delta^n(B')\Leftrightarrow\exists$ $\delta\geq0, \delta\in\mathbb{Z}^+$ so that $\frac{b_1}{r_1},\frac{b_2}{r_2}\in[\frac{\delta}{n},\frac{\delta+1}{n}]$
\item[(3)]$\sigma(B)=\sigma(B'); \sigma'(B)=\sigma'(B')+\frac{(r_1b_2-r_2b_1)^2}{r_1r_2(r_1+r_2)}$
\end{itemize}
\end{lem}

\begin{lem} (Chen--Chen \cite[Lemma 2.11]{Chen1}) 
 Let $B=\{(b_1,r_1),(b_2,r_2)\}\succeq{B'}=\{(b_1+b_2,r_1+r_2)\}$
 and $b_1r_2-b_2r_1=1$, set $n=r_1+r_2$. Then $\Delta^n(B)=\Delta^n(B')+1$.
\end{lem}


Recall that, in Chen--Chen \cite[2.13]{Chen1}, an operator ${\mathscr B}^{(n)}(\cdot)$ was defined to obtain the canonical sequence $\{{\mathscr B}^{(n)}(B)\}$ of the given basket $B$.  Roughly speaking,  ${\mathscr B}^{(n)}(B)$ is obtained by unpacking $B$ down to the level ``n''.  For $n>0$, set $B^{(n)}:={\mathscr B}^{(n)}(B)$. So have the canonical sequence:
$$B^{(0)}=\cdots=B^{(4)}\succeq B^{(5)}\succeq\cdots\overset{\alpha_n}\succeq B^{(n)}\overset{\alpha_{n+1}}\succeq\cdots\succeq B.$$
Define $\epsilon_n(B)$ to be the total number of prime packings belong to the packing process $\alpha_n$.

\begin{lem} (see Chen--Chen \cite[Claim B in 2.12, Lemma 2.16]{Chen1}) 
\begin{itemize}
\item[(1)] $
B^{(n-1)}=\mathscr{B}^{(n-1)}(\mathscr{B}^{(n)}(B))\succeq\mathscr{B}^{(n)}(B)=B^{(n)}
$
for all $n\geq1$.
\item[(2)]
$\Delta^j(B^{(0)})=\Delta^j(B)$ for $j=3,4$.
\item[(3)]
$\Delta^j(B^{(n-1)})=\Delta^j(B^{(n)})$ for $j<n$
\item[(4)]
$\Delta^n(B^{(n-1)})=\Delta^n(B^{(n)})+\epsilon_n(B)$.
\item[(5)]
$\Delta^n(B^{(n)})=\Delta^n(B)$.
\end{itemize}
\end{lem}

\subsection{The weighted basket}

A weighted basket is a formal triple ${\mathbb B}=\{B, \tilde{\chi}, \tilde{\chi}_2\}$ where $B$ is a generalized basket, $\tki\in {\mathbb Z}$ and $\tilde{\chi}_2\in {\mathbb Z}^+\cup \{0\}$.

We define the Euler characteristic and $K^3$ of a formal
basket formally as follows. First we
define
$$\left\{ \begin{array}{l} \chi_2({\mathbb B}):=\tilde{\chi}_2,\\ \chi_3({\mathbb B}):=-\sigma(B)+ 10
\tilde{\chi}+5\tilde{\chi}_2 \end{array} \right.$$ and the volume
$$\begin{array}{lll} K^3({\mathbb B})&:=&\sigma'({B})-4 \tilde{\chi}
-3\tilde{\chi}_2+\chi_3({\bf
B})\\
&=& -\sigma+\sigma'+6\tilde{\chi}+2\tilde{\chi}_2.
\end{array} 
$$ For $m \geq 4$, the Euler characteristic
$\chi_m({\mathbb B})$  is defined inductively by {\small
$$\chi_{m+1}({\mathbb B})-\chi_m({\mathbb B}):=
\frac{m^2}{2}(K^3({\mathbb B})-\sigma'({B})) +\frac{m}{2}
\sigma({B})-2\tilde{\chi}+\Delta^m(B).$$}

Clearly, by definition, $\chi_m({\mathbb B})$ is an integer for all
$m\geq 4$ because $K^3({\mathbb B})-\sigma'({B})=-4 \tilde{\chi}
-3\tilde{\chi}_2+\chi_3({\mathbb B})$ and $\sigma=10
\tilde{\chi}+5\tilde{\chi}_2-\chi_3({\mathbb B})$ have the same
parity.

Given a $\mathbb{Q}$-factorial terminal 3-fold $X$, one can associate to $X$ a triple ${\mathbb B}(X):=(B,\tilde{\chi}, \tilde{\chi_2}) $ where $B=B(X)$ is Reid's basket,
$\tilde{\chi}=\chi(\mathcal{O}_X)$ and $\tilde{\chi_2}= \chi(\mathcal{O}_X(2K_X))$. It's clear that such a triple is a formal basket. The Euler characteristic  and $K^3$ of the formal basket ${\mathbb B}(X)$ are nothing but
the Euler characteristic  and $K^3$ of the 3-fold $X$.
For simplicity, we denote $\chi_{m}({\mathbb B})$ by $\tilde{\chi}_m$
for all $m\geq 2$. Also denote $K^3({\mathbb B})$ by $\tilde{K}^3$.

\begin{defn} Let ${\mathbb B}:=(B, \tilde{\chi},
\tilde{\chi_2})$ and ${\bf B'}:=(B', \tilde{\chi},
\tilde{\chi_2})$ be two formal baskets.

(1) We say that ${\bf B'}$ is a packing of ${\mathbb B}$ (written as
${\mathbb B}\succ {\bf B'}$) if $B\succ B'$. Clearly ``packing'' between
formal baskets gives a partial ordering.

(2) A formal basket ${\mathbb B}$ is called {\it positive} if $K^3({\bf
B})>0$.

(3) A formal basket ${\mathbb B}$ is said to be minimal positive if it
is positive and minimal with regard to packing relation.
\end{defn}

In the case that $X$ is a minimal projective 3-fold of general type,  we see that 
$K^3({\mathbb B}(X))=K_X^3$ and $\tki_m=P_m(X)$ for all $m\geq 2$ by the vanishing theorem. 

\subsection{Representation of $B^{(n)}$ in terms of $\chi$ and $P_m$}
Given a minimal projective 3-fold $X$ of general type. Take ${\mathbb B}=\{B(X), \chi, P_2(X)\}$. Then $B^{(0)}$, $B^{(5)}$, $\cdots$, $B^{(12)}$ can be expressed in terms of $\chi$ and $P_m$ ($2\leq m\leq 13$). The readers are requested to check Chen--Chen \cite[P376--P379]{Chen1} for details which are omitted here due to controlling the bulk of this paper.

\section{\bf Proof of the main theorem}

\subsection{Boundedness}

Let $V$ be a nonsingular projective 3-fold of general type.
Pick up any minimal model $X$ of $V$.  Set $${\mathbb B}(X)=\{B(X), \chi(\mathcal{O}_X), P_2(X)\}.$$
We would like to prove that $\{{\mathbb B}(X)|P_2(X)=2\}$ is a finite set. 

\begin{lem}\label{b1} If $P_m\leq1$ for all $m\leq 11$ and $P_{12}=2$, then $P_2=0$. 
\end{lem}
\begin{proof} We shall repeatedly use those formulae in \cite[P376--P379]{Chen1} and keep the same notations as there.

{}First of all, we have
$$
\epsilon_6=-3P_2-P_3+P_4+P_5+P_6-P_7-\epsilon=0.
$$
Clearly we have $P_2\leq 1$.  If $P_2=1$, then $P_4=P_5=P_6=1$. It follows that $P_3=P_7=\epsilon=0$. But this is impossible since $P_2=P_5=1$ implies $P_7\geq1$. So $P_2=0$.
\end{proof}

So we have  the following equalities:
\begin{equation}
\begin{cases}
\tau:=\sigma'-K^3=4{\chi}-P_3\\
\sigma=10{\chi}-P_3\\
\Delta^3=5{\chi}-4P_3+P_4\\
\Delta^4=14{\chi}-6P_3-P_4+P_5\\
\Delta^5=27{\chi}-10P_3-P_5+P_6\\
\Delta^6=44{\chi}-15P_3-P_6+P_7\\
\Delta^7=65{\chi}-21P_3-P_7+P_8\\
\Delta^8=90{\chi}-28P_3-P_8+P_9\\
\Delta^9=119{\chi}-36P_3-P_9+P_{10}\\
\Delta^{10}=152{\chi}-45P_3-P_{10}+P_{11}\\
\Delta^{11}=189{\chi}-55P_3-P_{11}+P_{12}\\
\Delta^{12}=230{\chi}-66P_3-P_{12}+P_{13}
\end{cases}
\end{equation}

By Reid's Riemann-Roch formula, if $\chi(\mathcal{O}_X)\leq 0$, then $P_3\geq 2$ (contradicting to our assumption  of Theorem \ref{HLZ}).  By Chen--Chen \cite[Corollary3.13]{Chen2},  if $\chi(\mathcal{O}_X)=1$, then $P_{10}\geq 2$. 
Thus we may always assume $\chi(\mathcal{O}_X)>1$ from now on. 

\begin{lem}\label{b2}
If $P_m\leq1$ for all $~m\leq 11$ and $ P_{12}=2$, then $n_{1,r}^0=0$ for all $r\geq 7$.
\end{lem}
\begin{proof}
If there exists an integer $ r\geq7$ such that $n_{1,r}^0> 0$, then 
$R\geq 6$ and \cite[(3.14)]{Chen1} implies:
\begin{eqnarray*}
&&2P_5+3P_6+P_8+P_{10}
+P_{12}\\
&\geq&{\chi}+4P_3
+P_7+P_{11}+P_{13}+R.
\end{eqnarray*}
This implies $P_3=0$ and  $P_5=P_6=1$. Thus $P_{11}=1$. 
Eventually we see $P_8=P_{10}=1$ and $P_7=P_{13}=0$
But then $P_{13}\geq{P_5\cdot P_8=1}$ gives a contradiction. 
So we have $n_{1,r}^0=0$ for all $r\geq 7$. 
\end{proof}


\begin{lem}\label{b3} If $P_m\leq 1$ for all $m\leq11$, $P_{12}=2)$) and $n_{1,6}^0\neq 0$, then $n_{1,5}^0=0$ and $n_{1,6}^0=1$.
\end{lem}
\begin{proof} Clearly the definition of $R$ and the above inequality implies $n_{1,6}^0\leq 1$. 

If $n_{1,5}^0\neq 0$, we have 
$$
R=2n_{1,5}^0+5n_{1,6}^0\geq 7.
$$
So we get
$$
2P_5+3P_6+P_8+P_{10}\geq{P_7}+P_{11}+P_{13}+7,
$$
which says 
$P_5=P_6=P_8=P_{10}=1$ and $P_7=P_{11}=P_{13}=0.$
But then $P_{11}\geq P_5\cdot P_6=1$ (a contradiction).
Thus  $n_{1,5}^0=0$.
\end{proof}

Lemma \ref{b1}, Lemma \ref{b2} and Lemma \ref{b3} imply the following:

\begin{cor}\label{c} If $P_m\leq 1$ for all $m\leq11$ and  $P_{12}=2$, then one of the following occurs:
\begin{itemize}
\item[(i)] $n_{1,r}^0=0$ for all $r\geq6$.
\item[(ii)] $n_{1,5}^0=0$, $n_{1,6}^0=1$ and  $n_{1,r}^0=0$ for $r\geq7$.
\end{itemize}
\end{cor}

\subsection{Classification of Case (i)} By the definition of $R$ and $\epsilon$, we have  $R=2n_{1,5}^0$ and $\epsilon=n_{1,5}^0$.
To be precise, we have
$$
n_{1,5}^0=\epsilon=-P_3+P_4+P_5+P_6-P_7,
$$
$$
R=2\epsilon=-2P_3+2P_4+2P_5+2P_6-2P_7.
$$

We list all $B^{(n)}$ ($7\leq n\leq 12$) here to guide our computation. 
In explicit, we have
$$
\epsilon_7={\chi}-P_3+P_6+P_7-P_8
$$
and $B^{(7)}=\{n_{b,r}^7\times(b,r)\}_{\frac{b}{r}\in{S^{(7)}}}$ has its weights:
\begin{equation*}
B^{(7)}
\begin{cases}
n^7_{1,2}=2{\chi}-3P_3+3P_4-P_5+P_6-2P_7+P_8+\eta\\
n^7_{3,7}={\chi}-P_3+P_6+P_7-P_8-\eta\\
n^7_{2,5}={\chi}+P_3-P_4+P_5-3P_6+P_8+\eta\\
n^7_{1,3}=2{\chi}+2P_3-2P_4+2P_6-P_7-\eta\\
n^7_{2,7}=\eta\\
n^7_{1,4}={\chi}+2P_3+P_4-2P_5-P_6+P_7-\eta\\
n^7_{1,5}=-P_3+P_4+P_5+P_6-P_7
\end{cases}
\end{equation*}
where $\eta$ is the number of prime packings of type $\{(1,3), (1,4)\}\succ \{(2,7)\}$. 
We have already known
$$
\epsilon_8=-P_3-P_4+P_5+P_6+P_8-P_9
$$
Thus, taking some prime packings into consideration,
 $B^{(8)}=\{n_{b,r}^8\times(b,r)\}_{\frac{b}{r}\in{S^{(8)}}}$ has coefficients:
\begin{equation*}
B^{(8)}
\begin{cases}
n^8_{1,2}=2{\chi}-3P_3+2P_4-P_5+P_6-2P_7+P_8+\eta\\
n^8_{3,7}={\chi}-P_3+P_6+P_7-P_8-\eta\\
n^8_{2,5}={\chi}+2P_3-4P_6+P_9+\eta\\
n^8_{3,8}=-P_3-P_4+P_5+P_6+P_8-P_9\\
n^8_{1,3}=2{\chi}+3P_3-P_4-P_5+P_6-P_7-P_8+P_9-\eta\\
n^8_{2,7}=\eta\\
n^8_{1,4}={\chi}+2P_3+P_4-2P_5-P_6+P_7-\eta\\
n^8_{1,5}=-P_3+P_4+P_5+P_6-P_7
\end{cases}
\end{equation*}
We know that:
$$
\epsilon_9=-P_3-P_6+P_8+P_9-P_{10}+\eta
$$
 Moreover $S^{(9)}-S^{(8)}=\{\frac{4}{9},\frac{2}{9}\}$.
 Let $\zeta$ be the number of prime packing of type
 $\{(1,2),(3,7)\}\succ\{(4,9)\}$, then the number of
 type $\{(1,4),(1,5)\}\succ\{(2,9)\}$ prime packing
 is $\epsilon_9-\zeta$. We can get $B^{(9)}$ consisting
 of the following coefficients:
\begin{equation*}
B^{(9)}
\begin{cases}
n^9_{1,2}=2{\chi}-3P_3+2P_4-P_5+P_6-2P_7+P_8+\eta-\zeta\\
n^9_{4,9}=\zeta\\
n^9_{3,7}={\chi}-P_3+P_6+P_7-P_8-\eta-\zeta\\
n^9_{2,5}={\chi}+2P_3-4P_6+P_9+\eta\\
n^9_{3,8}=-P_3-P_4+P_5+P_6+P_8-P_9\\
n^9_{1,3}=2{\chi}+3P_3-P_4-P_5+P_6-P_7-P_8+P_9-\eta\\
n^9_{2,7}=\eta\\
n^9_{1,4}={\chi}+3P_3+P_4-2P_5-P_6+P_7-P_8-P_9+P_{10}-2\eta+\zeta\\
n^9_{2,9}=-P_3-P_6+P_8+P_9-P_{10}+\eta-\zeta\\
n^9_{1,5}=P_4+P_5+2P_6-P_7-P_8-P_9+P_{10}-\eta+\zeta
\end{cases}
\end{equation*}
One has
$$
\epsilon_{10}=-P_4-P_5+P_6+P_7+P_{10}-P_{11}-\eta
$$
and then $B^{(10)}$ consists of the following coefficients:
\begin{equation*}
B^{(10)}
\begin{cases}
n^{10}_{1,2}=2{\chi}-3P_3+2P_4-P_5+P_6-2P_7+P_8+\eta-\zeta\\
n^{10}_{4,9}=\zeta\\
n^{10}_{3,7}={\chi}-P_3+P_6+P_7-P_8-\eta-\zeta\\
n^{10}_{2,5}={\chi}+2P_3-4P_6+P_9+\eta\\
n^{10}_{3,8}=-P_3-P_4+P_5+P_6+P_8-P_9\\
n^{10}_{1,3}=2{\chi}+3P_3-2P_7-P_8+P_9-P_{10}+P_{11}-\eta\\
n^{10}_{3,10}=-P_4-P_5+P_6+P_7+P_{10}-P_{11}-\eta\\
n^{10}_{2,7}=P_4+P_5-P_6-P_7-P_{10}+P_{11}+2\eta\\
n^{10}_{1,4}={\chi}+3P_3+P_4-2P_5+P_7-P_8-P_9+P_{10}-2\eta+\zeta\\
n^{10}_{2,9}=-P_3-P_6+P_8+P_9-P_{10}+\eta-\zeta\\
n^{10}_{1,5}=P_4+P_5+2P_6-P_7-P_8-P_9+P_{10}-\eta+\zeta
\end{cases}
\end{equation*}
By computing $\Delta^{11}(B^{(10)})$, I get
$$
\epsilon_{11}={\chi}-P_5-P_6+P_9+P_{11}-P_{12}-\zeta
$$
 Let $\alpha$ be the number of prime packing of
 type $\{(1,2),(4,9)\}\succ\{(5,11)\}$, and $\beta$
 be the number of prime packing of type $\{(1,3),(3,8)\}\succ\{(4,11)\}$.
 Then, I get $B^{(11)}$ with
\begin{equation*}
B^{(11)}
\begin{cases}
n^{11}_{1,2}=2{\chi}-3P_3+2P_4-P_5+P_6-2P_7+P_8+\eta-\zeta-\alpha\\
n^{11}_{5,11}=\alpha\\
n^{11}_{4,9}=\zeta-\alpha\\
n^{11}_{3,7}={\chi}-P_3+P_6+P_7-P_8-\eta-\zeta\\
n^{11}_{2,5}={\chi}+2P_3-4P_6+P_9+\eta\\
n^{11}_{3,8}=-P_3-P_4+P_5+P_6+P_8-P_9-\beta\\
n^{11}_{4,11}=\beta\\
n^{11}_{1,3}=2{\chi}+3P_3-2P_7-P_8+P_9-P_{10}+P_{11}-\eta-\beta\\
n^{11}_{3,10}=-P_4-P_5+P_6+P_7+P_{10}-P_{11}-\eta\\
n^{11}_{2,7}=-{\chi}+P_4+2P_5-P_7-P_9-P_{10}+P_{12}+2\eta+\zeta+\alpha+\beta\\
n^{11}_{3,11}={\chi}-P_5-P_6+P_9+P_{11}-P_{12}-\zeta-\alpha-\beta\\
n^{11}_{1,4}=3P_3+P_4-P_5+P_6+P_7-P_8-2P_9+P_{10}-P_{11}+P_{12}-2\eta+2\zeta+\alpha+\beta\\
n^{11}_{2,9}=-P_3-P_6+P_8+P_9-P_{10}+\eta-\zeta\\
n^{11}_{1,5}=P_4+P_5+2P_6-P_7-P_8-P_9+P_{10}-\eta+\zeta
\end{cases}
\end{equation*}
Finally, since
$$
\epsilon_{12}=-{\chi}-2P_3-P_4+P_5+P_8+P_{12}-P_{13}+\eta
$$
\begin{equation*}
B^{(12)}
\begin{cases}
n^{12}_{1,2}=2{\chi}-3P_3+2P_4-P_5+P_6-2P_7+P_8+\eta-\zeta-\alpha\\
n^{12}_{5,11}=\alpha\\
n^{12}_{4,9}=\zeta-\alpha\\
n^{12}_{3,7}=2{\chi}+P_3+P_4-P_5+P_6+P_7-2P_8-P_{12}+P_{13}-2\eta-\zeta\\
n^{12}_{5,12}=-{\chi}-2P_3-P_4+P_5+P_8+P_{12}-P_{13}+\eta\\
n^{12}_{2,5}=2{\chi}+4P_3+P_4-P_5-4P_6-P_8+P_9-P_{12}+P_{13}\\
n^{12}_{3,8}=-P_3-P_4+P_5+P_6+P_8-P_9-\beta\\
n^{12}_{4,11}=\beta\\
n^{12}_{1,3}=2{\chi}+3P_3-2P_7-P_8+P_9-P_{10}+P_{11}-\eta-\beta\\
n^{12}_{3,10}=-P_4-P_5+P_6+P_7+P_{10}-P_{11}-\eta\\
n^{12}_{2,7}=-{\chi}+P_4+2P_5-P_7-P_9-P_{10}+P_{12}+2\eta+\zeta+\alpha+\beta\\
n^{12}_{3,11}={\chi}-P_5-P_6+P_9+P_{11}-P_{12}-\zeta-\alpha-\beta\\
n^{12}_{1,4}=3P_3+P_4-P_5+P_6+P_7-P_8-2P_9+P_{10}-P_{11}+P_{12}-2\eta+2\zeta+\alpha+\beta\\
n^{12}_{2,9}=-P_3-P_6+P_8+P_9-P_{10}+\eta-\zeta\\
n^{12}_{1,5}=P_4+P_5+2P_6-P_7-P_8-P_9+P_{10}-\eta+\zeta
\end{cases}
\end{equation*}

We now explain the effectivity of the possible computation. We have $\epsilon_{10}+\epsilon_{12}\geq 0$ which gives rise to:
\begin{equation}
2P_5+3P_6+P_8+P_{10}+2\geq{\chi}+4P_3+P_7+P_{11}+P_{13}+R
\end{equation}
Inserting the expression of $R$, one has:
$$
{\chi}\leq-2P_3-2P_4+P_6+P_7+P_8+P_{10}-P_{11}-P_{13}+2.
$$
If the equality holds, one has
$
P_3=P_4=P_{11}=P_{13}=0$ and $P_6=P_7=P_8=P_{10}=1.$
But then $P_{13}\geq{P_6\cdot P_7=1}$, a contradiction. Thus we have 
\begin{equation}
{\chi}\leq5
\end{equation}
Considering $n_{1,4}^7\geq 0$, we get:
$$
\eta\leq{\chi}+2P_3+P_4-2P_5-P_6+P_7.
$$
So,
\begin{equation}
\eta\leq{\chi}+4
\end{equation}
We have $n_{2,9}^7\geq 0$, which implies:
$$
\zeta\leq{\chi}+P_3+P_4-2P_5-2P_6+P_7+P_8+P_9-P_{10}.
$$
Then,
\begin{equation}
\zeta\leq{\chi}+4.
\end{equation}
Similarly, by $n_{4,9}^{11}\geq0$ and $n_{3,8}^{11}\geq 0$, we get:
\begin{equation}
\alpha\leq\zeta
\end{equation}
\begin{equation}
\beta\leq3
\end{equation}
Finally we have:
\begin{equation}
P_{13}\leq-{\chi}-2P_3-2P_4+P_6+P_7+P_8+P_{10}-P_{11}+2\leq 
-{\chi}+6.
\end{equation}

Inequalities $(3.3)\sim (3.8)$ tell us that $\{B^{(12)}\}$ is a computable finite set. 
Our computer program using Matlab outputs totally 
 $149$ classes as follows. We also list all possible minimal positive baskets dominated by these 149 classes of $B^{(12)}$.  The last volume provides their volumes. 

\par\nobreak
\LTleft=-10pt \LTright=0pt
\begin{center}
\scriptsize
\begin{longtable}{p{4mm}p{25mm}p{80mm}p{10mm}}
\noalign{\hrule height .08em\vskip.8ex}
No & $({\chi},P_3,P_4,\ldots,P_{11})$ & $B^{(12)}=(n_{1,2}^{12},n_{5,11}^{12},\ldots,n_{1,5}^{12})$ & $K^3$\quad \\
\noalign{\vskip .4ex \hrule height 0.05em\vskip.65ex}
1 & 2,0,0,0,0,0,0,0,0,0 &4,0,0,2,0,2,0,0,4,0,0,0,2,0,0 &1/210 \\
2 & 2,0,1,0,0,0,1,0,1,0 &7,0,0,1,0,2,0,0,2,0,0,0,3,0,1 &1/420 \\
3 & 2,0,0,1,0,0,0,1,1,0 &3,0,0,1,1,2,0,0,4,0,0,0,0,0,1 &1/420\\
a &        & (3,7),(5,12)$\succ$(8,19) & 1/570 \\
b &        & (5,12),2(2,5)$\succ$(9,22) & 1/1155 \\
4 & 2,0,0,1,0,0,0,1,1,0 & 3,0,0,2,0,3,0,0,4,0,0,0,0,0,1 & 1/210 \\
5 & 2,0,0,1,0,1,1,0,1,1 & 2,0,0,1,1,1,2,0,1,0,0,0,1,0,0 & 1/420 \\
a &        & (3,7),(5,12)$\succ$(8,19) & 1/570 \\
b &        & (5,12),(2,5)$\succ$(7,17) & 1/714 \\
6 & 2,0,0,1,0,1,1,0,1,1 & 2,0,0,2,0,2,2,0,1,0,0,0,1,0,0 & 1/210 \\
a &        & 2(2,5),2(3,8)$\succ$2(5,13) & 1/1092 \\
7 & 2,0,0,0,0,0,1,0,1,1 & 5,0,0,0,1,1,0,1,2,0,0,0,2,0,0 & 1/220 \\
a &        & (5,12),(2,5)$\succ$(7,17);(4,11),(1,3)$\succ$(5,14) & 1/714 \\
b &        & (4,11),2(1,3)$\succ$(6,17) & 1/1020 \\
8 , & 2,0,0,0,0,0,1,0,1,1 & 5,0,0,1,0,2,0,1,2,0,0,0,2,0,0 & 8/1155 \\
a &        & (4,11),2(1,3)$\succ$(6,17) & 2/595 \\
9 & 2,0,0,1,0,1,1,1,1,1 & 1,0,1,0,1,2,1,0,2,0,0,0,1,0,0 & 1/360 \\
a &        & (5,12),2(2,5)$\succ$(9,22) & 1/792 \\
b& &(2,5),(3,8)$\succ$(5,13)&1/1170\\
10 & 2,0,0,1,0,1,1,1,1,1 & 1,0,1,1,0,3,1,0,2,0,0,0,1,0,0 & 13/2520 \\
a &        & (4,9),(3,7)$\succ$(7,16);(6,15),(3,8)$\succ$(9,23) & 1/1104 \\
11 & 2,0,0,0,0,0,1,1,1,1 & 3,1,0,0,0,3,0,0,4,0,0,0,2,0,0 & 1/165 \\
a& &(1,2),(5,11)$\succ$(6,13)&1/390\\
12 & 2,0,0,0,0,0,0,0,1,0 & 5,0,0,0,1,2,0,0,3,0,1,0,1,0,0 & 1/210 \\
a &        & (5,12),2(2,5)$\succ$(9,22) & 1/308 \\
13 & 2,0,0,0,0,0,0,0,1,0 & 5,0,0,1,0,3,0,0,3,0,1,0,1,0,0 & 1/140 \\
14 & 2,0,0,0,1,0,0,1,1,0 & 6,0,0,2,0,0,0,0,4,1,0,0,0,0,1 & 1/210 \\
a &        & (1,3),(3,10)$\succ$(4,13) & 1/455 \\
15 & 2,0,1,0,1,0,1,1,1,0 & 9,0,0,1,0,0,0,0,3,0,1,0,0,1,1 & 1/630 \\
16 & 2,0,1,0,1,1,1,1,1,1 & 7,0,0,2,0,0,0,0,2,0,0,1,0,1,0 & 1/1386 \\
17 & 2,0,0,0,1,1,0,1,0,1 & 3,0,1,2,0,0,0,0,4,0,1,0,1,0,0 & 1/252 \\
a &        & (4,9),2(3,7)$\succ$(10,23) & 5/1932 \\
18 & 2,0,1,0,1,1,1,1,1,1 & 6,0,1,1,0,0,0,0,2,0,1,0,2,0,1 & 1/630 \\
19 & 2,0,0,0,1,0,1,1,1,1 & 6,0,1,0,0,0,1,0,4,0,1,0,0,0,1 & 19/2520 \\
20 & 2,0,0,0,1,1,1,1,1,1 & 4,0,1,1,0,0,1,0,2,1,0,0,1,0,0 & 13/2520 \\
a &        & (4,9),(3,7)$\succ$(7,16);(1,3),(3,10)$\succ$(4,13) & 1/624 \\
b &        & 2(1,3),(3,10)$\succ$(5,16) & 1/1008 \\
21 & 2,0,0,0,1,1,0,0,1,1 & 5,0,0,2,0,0,1,0,2,0,2,0,0,0,0 & 1/168 \\
22 & 2,0,0,0,1,1,0,1,1,1 & 4,0,1,0,1,0,0,0,3,0,2,0,0,0,0 & 1/252 \\
23 & 2,0,0,0,1,1,0,1,1,1 & 4,0,1,1,0,1,0,0,3,0,2,0,0,0,0 & 2/315 \\
a &        & (4,9),(3,7)$\succ$(7,16) & 3/560 \\
24 & 3,0,0,1,0,0,1,0,1,0 & 6,0,0,2,0,3,2,0,4,0,0,0,1,0,1 & 1/210 \\
a &        & (4,10),2(3,8)$\succ$(10,26) & 1/1092 \\
25 & 3,0,0,1,0,0,1,1,1,0 & 6,0,0,2,0,4,0,1,4,0,0,0,0,1,0 & 13/3465 \\
a &        & (4,11),(1,3)$\succ$(5,14) & 1/630 \\
26 & 3,0,0,1,0,1,1,0,1,1 & 4,0,0,3,0,3,1,1,2,0,0,0,2,0,0 & 31/9240 \\
a &        & (2,5),(3,8)$\succ$(5,13) & 43/30030 \\
b &        & (4,11),(1,3)$\succ$(5,14) & 1/840 \\
27 & 3,0,0,1,1,0,1,1,1,1 & 7,0,0,3,0,0,1,1,5,0,0,0,0,0,2 & 31/9240 \\
a &        & (3,8),(4,11)$\succ$(7,19) & 11/3990 \\
b &        & (4,11),(1,3)$\succ$(5,14) & 1/840 \\
28 & 3,0,0,0,0,0,1,0,0,0 & 6,0,1,1,0,3,1,0,5,0,0,0,3,0,0 & 13/2520 \\
a &        & (4,9),(3,7)$\succ$(7,16);(6,15),(3,8)$\succ$(9,23) & 1/1104 \\
29 & 3,0,0,1,0,0,1,1,1,0 & 5,0,1,1,0,4,1,0,5,0,0,0,1,0,1 & 13/2520 \\
a &        & (4,9),(3,7)$\succ$(7,16);(6,15),(3,8)$\succ$(9,23) & 1/1104 \\
b &        & (8,20),(3,8)$\succ$(11,28) & 1/630 \\
30 & 3,0,0,0,1,0,1,1,0,1 & 7,0,1,2,0,0,0,1,6,0,0,0,2,0,1 & 13/3465 \\
a &        & (4,9),(6,14)$\succ$(10,23) & 3/1265 \\
b &        & (4,11),(1,3)$\succ$(5,14) & 1/630 \\
31 & 3,0,0,1,0,1,1,1,1,1 & 3,0,1,2,0,4,0,1,3,0,0,0,2,0,0 & 13/3465 \\
a &        & (4,9),(6,14)$\succ$(10,23) & 3/1265 \\
b &        & (4,11),(1,3)$\succ$(5,14) & 1/630 \\
32 & 3,0,0,0,1,0,1,1,0,1 & 6,1,0,2,0,0,1,0,7,0,0,0,2,0,1 & 23/9240 \\
33 & 3,0,0,1,0,1,1,1,1,1 & 2,1,0,2,0,4,1,0,4,0,0,0,2,0,0 & 23/9240 \\
34 & 3,0,0,0,0,0,1,1,0,0 & 5,0,2,0,0,4,0,0,6,0,0,0,3,0,0 & 1/180 \\
35 & 3,0,0,0,1,0,0,0,0,0 & 8,0,0,3,0,0,1,0,6,0,1,0,1,0,1 & 1/280 \\
36 & 3,0,0,0,0,0,0,0,1,0 & 7,0,0,2,0,4,0,0,5,0,0,1,1,0,0 & 29/4620 \\
a &        & (3,11),(1,4)$\succ$(4,15) & 1/210 \\
37 & 3,0,0,1,0,1,0,0,1,0 & 4,0,0,2,1,3,1,0,3,0,1,0,1,0,0 & 1/840 \\
38 & 3,0,1,0,1,0,1,0,1,0 & 11,0,0,2,0,0,1,0,4,0,1,0,2,0,2 & 1/840 \\
39 & 3,0,0,0,1,0,1,0,0,0 & 9,0,0,2,0,0,2,0,5,0,1,0,0,1,0 & 1/252 \\
40 & 3,0,0,1,0,1,0,0,1,0 & 4,0,0,3,0,4,1,0,3,0,1,0,1,0,0 & 1/280 \\
a &        & (2,5),(3,8)$\succ$(5,13) & 3/1820 \\
41 & 3,0,0,0,1,0,1,0,1,0 & 9,0,0,2,0,0,2,0,4,1,0,0,1,0,1 & 1/210 \\
a &        & (1,3),(3,10)$\succ$(4,13) & 1/455 \\
42 & 3,0,0,0,1,0,1,0,1,1 & 9,0,0,2,0,0,2,0,5,0,0,1,0,0,1 & 29/4620 \\
43 & 3,0,0,1,1,1,1,0,1,1 & 6,0,0,3,0,0,3,0,3,0,1,0,0,0,1 & 1/280 \\
44 & 3,0,0,0,1,0,1,1,1,0 & 9,0,0,1,1,0,0,1,4,1,0,0,0,1,0 & 19/13860 \\
a &        & (3,7),(5,12)$\succ$(8,19) & 7/9405 \\
45 & 3,0,0,0,1,0,1,1,1,0 & 9,0,0,2,0,1,0,1,4,1,0,0,0,1,0 & 13/3465 \\
a &        & (4,11),(1,3)$\succ$(5,14) & 1/630 \\
b &        & (1,3),(3,10)$\succ$(4,13) & 27/22733 \\
46 & 3,0,0,0,1,0,1,0,1,1 & 9,0,0,2,0,0,1,1,4,0,1,0,1,0,1 & 53/9240 \\
a &        & (3,8),(4,11),(1,3)$\succ$(8,22) & 3/1540 \\
b &        & (4,11),3(1,3)$\succ$(7,20) & 1/840 \\
47 & 3,0,0,0,1,1,1,0,1,1 & 7,0,0,3,0,0,1,1,2,1,0,0,2,0,0 & 31/9240 \\
a &        & (3,8),(4,11)$\succ$(7,19) & 11/3990 \\
b &        & (4,11),(1,3)$\succ$(5,14) & 1/840 \\
c &        & (1,3),(3,10)$\succ$(4,13) & 19/24024 \\
48 & 3,0,0,1,1,1,0,1,1,1 & 5,0,0,4,0,1,0,1,4,0,1,0,0,0,1 & 1/462 \\
49 & 3,0,0,0,1,0,0,1,0,0 & 7,0,1,2,0,1,0,0,7,0,1,0,1,0,1 & 1/252 \\
a &        & (4,9),2(3,7)$\succ$(10,23) & 5/1932 \\
50 & 3,0,0,0,1,0,1,1,0,0 & 8,0,1,0,1,0,1,0,6,0,1,0,0,1,0 & 1/504 \\
51 & 3,0,0,1,0,1,0,1,1,0 & 3,0,1,1,1,4,0,0,4,0,1,0,1,0,0 & 1/630 \\
a &        & (3,7),(5,12)$\succ$(8,19) & 23/23940 \\
52 & 3,0,0,0,1,0,1,1,1,0 & 8,0,1,0,1,0,1,0,5,1,0,0,1,0,1 & 1/360 \\
53 & 3,0,0,0,1,0,1,1,1,1 & 8,0,1,0,1,0,1,0,6,0,0,1,0,0,1 & 17/3960 \\
54 & 3,0,0,0,0,0,1,0,1,0 & 7,0,1,0,0,4,1,0,4,0,1,0,2,0,0 & 19/2520 \\
a &        & 4(2,5),(3,8)$\succ$(11,28) & 1/252 \\
55 & 3,0,0,0,1,0,1,1,0,0 & 8,0,1,1,0,1,1,0,6,0,1,0,0,1,0 & 11/2520 \\
a &        & (4,9),(3,7)$\succ$(7,16);(2,5),(3,8)$\succ$(5,13) & 19/13104 \\
56 & 3,0,0,1,0,1,0,1,1,0 & 3,0,1,2,0,5,0,0,4,0,1,0,1,0,0 & 1/252 \\
a &        & (4,9),2(3,7)$\succ$(10,23) & 5/1932 \\
57 & 3,0,0,0,1,1,1,0,0,1 & 6,0,1,2,0,0,2,0,4,0,1,0,2,0,0 & 1/252 \\
a &        & (4,9),2(3,7)$\succ$(10,23) & 5/1932 \\
58 & 3,0,0,0,1,0,1,1,1,0 & 8,0,1,1,0,1,1,0,5,1,0,0,1,0,1 & 13/2520 \\
a &        & (4,9),(3,7)$\succ$(7,16);(2,5),(3,8)$\succ$(5,13) & 7/3120 \\
b &        & 2(1,3),(3,10)$\succ$(5,16) & 1/1008 \\
59 & 3,0,0,0,1,0,1,1,1,1 & 8,0,1,1,0,1,1,0,6,0,0,1,0,0,1 & 37/5544 \\
a &        & (4,9),(3,7)$\succ$(7,16);(2,5),(3,8)$\succ$(5,13) & 43/11440 \\
60 & 3,0,0,1,1,1,1,1,1,1 & 5,0,1,1,1,0,2,0,4,0,1,0,0,0,1 & 1/630 \\
a &        & (3,7),(5,12)$\succ$(8,19) & 23/23940 \\
61 & 3,0,0,1,1,1,1,1,1,1 & 5,0,1,2,0,1,2,0,4,0,1,0,0,0,1 & 1/252 \\
a &        & (4,9),2(3,7)$\succ$(10,23) & 5/1932 \\
b &        & (2,5),2(3,8)$\succ$(8,21) & 1/630 \\
62 & 3,0,0,0,1,0,1,1,1,1 & 8,0,1,0,1,0,0,1,5,0,1,0,1,0,1 & 13/3465 \\
a &        & (4,11),(1,3)$\succ$(5,14) & 1/630 \\
63 & 3,0,0,0,1,0,1,1,1,1 & 8,0,1,1,0,1,0,1,5,0,1,0,1,0,1 & 17/2772 \\
a &        & (4,9),(3,7)$\succ$(7,16);(4,11),2(1,3)$\succ$(6,17) & 3/1904 \\
b &        & (4,11),4(1,3)$\succ$(8,23) & 5/5796 \\
64 & 3,0,0,0,1,1,1,1,1,1 & 6,0,1,2,0,1,0,1,3,1,0,0,2,0,0 & 13/3465 \\
a &        & (4,9),2(3,7)$\succ$(10,23) & 3/1265 \\
b , &        & (4,11),(1,3)$\succ$(5,14) & 1/630 \\
c &        & (1,3),(3,10)$\succ$(4,13) & 27/22733 \\
65 & 3,0,0,0,1,0,1,1,1,1 & 7,1,0,0,1,0,1,0,6,0,1,0,1,0,1 & 23/9240 \\
66 & 3,0,0,0,1,0,1,1,1,1 & 7,1,0,1,0,1,1,0,6,0,1,0,1,0,1 & 3/616 \\
a &        & (2,5),(3,8)$\succ$(5,13) & 59/20020 \\
67 & 3,0,0,0,1,1,1,1,1,1 & 5,1,0,2,0,1,1,0,4,1,0,0,2,0,0 & 23/9240 \\
68 & 3,0,0,0,1,1,1,1,0,1 & 5,0,2,1,0,1,1,0,5,0,1,0,2,0,0 & 11/2520 \\
a &        & 2(4,9),(3,7)$\succ$(11,25);(2,5),(3,8)$\succ$(5,13) & 8/6825 \\
69 & 3,0,0,0,1,0,0,0,1,0 & 9,0,0,1,1,0,1,0,5,0,2,0,0,0,1 & 1/280 \\
a &        & (3,7),(5,12)$\succ$(8,19) & 47/15960 \\
70 & 3,0,0,0,1,0,0,0,1,0 & 9,0,0,2,0,1,1,0,5,0,2,0,0,0,1 & 1/168 \\
a &        & (2,5),(3,8)$\succ$(5,13) & 11/2730 \\
71 & 3,0,0,0,1,1,0,0,1,0 & 7,0,0,3,0,1,1,0,3,1,1,0,1,0,0 & 1/280 \\
a &        & (2,5),(3,8)$\succ$(5,13);(3,10),(2,7)$\succ$(5,17) & 5/6188 \\
b &        & (1,3),(3,10)$\succ$(4,13) & 11/10920 \\
72 & 3,0,0,0,1,1,0,0,1,1 & 7,0,0,3,0,1,1,0,4,0,1,1,0,0,0 & 47/9240 \\
a &        & (2,5),(3,8)$\succ$(5,13);(2,7),(3,11)$\succ$(5,18) & 2/819 \\
73 & 3,0,0,0,1,1,0,0,1,1 & 7,0,0,3,0,1,0,1,3,0,2,0,1,0,0 & 1/220 \\
&        & (4,11),2(1,3)$\succ$(6,17) & 1/1020 \\
74 & 3,0,0,0,1,0,0,1,1,0 & 8,0,1,0,1,1,0,0,6,0,2,0,0,0,1 & 1/252 \\
a &        & (5,12),(2,5)$\succ$(7,17) & 16/5355 \\
75 & 3,0,0,0,1,0,0,1,1,0 & 8,0,1,1,0,2,0,0,6,0,2,0,0,0,1 & 2/315 \\
a &        & (4,9),(3,7)$\succ$(7,16) & 3/560 \\
76 & 3,0,0,0,1,1,0,1,1,0 & 6,0,1,2,0,2,0,0,4,1,1,0,1,0,0 & 1/252 \\
a &        & (4,9),2(3,7)$\succ$(10,23);(3,10),(2,7)$\succ$(5,17) & 41/23460 \\
b &        & (1,3),(3,10)$\succ$(4,13) & 23/16380 \\
77 & 3,0,0,0,1,1,1,0,1,1 & 7,0,1,0,1,0,2,0,3,0,2,0,1,0,0 & 1/252 \\
78 & 3,0,0,0,1,1,0,1,1,1 & 6,0,1,2,0,2,0,0,5,0,1,1,0,0,0 & 19/3465 \\
a &        & (4,9),2(3,7)$\succ$(10,23);(2,7),(3,11)$\succ$(5,18) & 7/2070 \\
79 & 3,0,0,0,1,1,1,0,1,1 & 7,0,1,1,0,1,2,0,3,0,2,0,1,0,0 & 2/315 \\
a &        & (4,9),(3,7)$\succ$(7,16);(2,5),2(3,8)$\succ$(8,21) & 1/336 \\
80 & 3,0,0,0,1,1,0,1,1,1 & 5,1,0,2,0,2,0,0,5,0,2,0,1,0,0 & 17/4620 \\
81 & 3,0,0,0,1,1,1,1,1,1 & 6,0,2,0,0,2,1,0,4,0,2,0,1,0,0 & 17/2520 \\
a &        & 2(2,5),(3,8)$\succ$(7,18) & 1/252 \\
82 & 4,0,0,0,1,0,1,0,0,0 & 11,0,0,3,0,1,2,0,7,0,0,1,0,1,0 & 43/13860 \\
a &        & (2,5),2(3,8)$\succ$(8,21) & 1/1386 \\
83 & 4,0,0,1,0,1,0,0,1,0 & 6,0,0,4,0,5,1,0,5,0,0,1,1,0,0 & 5/1848 \\
a &        & (3,11),(1,4)$\succ$(4,15) & 1/840 \\
b &        & (2,5),(3,8)$\succ$(5,13) & 47/60060 \\
84 & 4,0,0,1,1,1,1,0,1,1 & 8,0,0,4,0,1,3,0,5,0,0,1,0,0,1 & 5/1848 \\
a &        & (2,5),(3,8)$\succ$(5,13) & 47/60060 \\
85 & 4,0,0,0,1,0,1,0,0,0 & 11,0,0,3,0,1,1,1,6,0,1,0,1,1,0 & 71/27720 \\
a &        & (3,8),(4,11)$\succ$(7,19) & 47/23940 \\
86 & 4,0,0,1,0,1,0,0,1,0 & 6,0,0,4,0,5,0,1,4,0,1,0,2,0,0 & 1/462 \\
87 & 4,0,0,0,1,0,1,0,1,0 & 11,0,0,3,0,1,1,1,5,1,0,0,2,0,1 & 31/9240 \\
a &        & (2,5),(3,8)$\succ$(5,13) & 43/30030 \\
b &        & (3,8),(4,11)$\succ$(7,19) & 11/3990 \\
c &        & (4,11),(1,3)$\succ$(5,14) & 1/840 \\
d &        & (1,3),(3,10)$\succ$(4,13) & 19/24024 \\
88 & 4,0,0,0,1,0,1,0,1,1 & 11,0,0,3,0,1,1,1,6,0,0,1,1,0,1 & 3/616 \\
a &        & (3,11),(1,4)$\succ$(4,15);(2,5),(3,8)$\succ$(5,13) & 43/30030 \\
b &        & (3,11),(1,4)$\succ$(4,15);(3,8),(4,11)$\succ$(7,19) & 11/3990 \\
c &        & (3,11),(1,4)$\succ$(4,15);(4,11),(1,3)$\succ$(5,14) & 1/840 \\
d &        & (2,5),(3,8)$\succ$(5,13);(4,11),(1,3)$\succ$(5,14) & 47/60060 \\
e &        & (4,11),2(1,3)$\succ$(6,17) & 41/31416 \\
89 & 4,0,0,1,1,1,1,0,1,1 & 8,0,0,4,0,1,2,1,4,0,1,0,1,0,1 & 1/462 \\
a &        & 2(3,8),(4,11)$\succ$(10,27) & 1/756 \\
90 & 4,0,0,0,1,0,1,0,1,1 & 11,0,0,3,0,1,0,2,5,0,1,0,2,0,1 & 1/231 \\
a &        & 2(4,11),(1,3)$\succ$(9,25) & 1/525 \\
b &        & (4,11),2(1,3)$\succ$(6,17) & 1/1309 \\
91 & 4,0,0,1,1,1,1,1,1,1 & 8,0,0,4,0,2,0,2,4,0,1,0,0,1,0 & 4/3465 \\
92 & 4,0,0,0,1,0,1,0,0,0 & 10,0,1,2,0,1,2,0,7,0,1,0,2,0,1 & 1/252 \\
a &        & (2,5),2(3,8)$\succ$(8,21) & 1/630 \\
b &        & (4,9),2(3,7)$\succ$(10,23) & 5/1932 \\
93 & 4,0,0,0,0,0,1,0,1,0 & 9,0,1,1,0,5,1,0,6,0,0,1,2,0,0 & 37/5544 \\
a &        & (4,9),(3,7)$\succ$(7,16);(3,11),2(1,4)$\succ$(5,19);(2,5),(3,8)$\succ$(5,13) & 27/19760 \\
b &        & (4,9),(3,7)$\succ$(7,16);5(2,5),(3,8)$\succ$(13,33) & 1/528 \\
c &        & 5(2,5),(3,8)$\succ$(13,33);(3,11),(1,4)$\succ$(4,15) & 19/13860 \\
d &        & (3,11),2(1,4)$\succ$(5,19),3(2,5),(3,8)$\succ$(9,23) & 4/3919 \\
94 & 4,0,0,0,1,0,1,1,0,0 & 10,0,1,2,0,2,1,0,8,0,0,1,0,1,0 & 97/27720 \\
a &        & 2(2,5),(3,8)$\succ$(7,18) & 1/1386 \\
b &        & (4,9),2(3,7)$\succ$(10,23) & 12/5663 \\
95 & 4,0,0,1,0,1,0,1,1,0 & 5,0,1,3,0,6,0,0,6,0,0,1,1,0,0 & 43/13860 \\
a &        & (3,11),(1,4)$\succ$(4,15) & 1/630 \\
b &        & (4,9),3(3,7)$\succ$(13,30) & 1/660 \\
96 & 4,0,0,1,0,1,1,0,1,0 & 6,0,1,2,0,5,2,0,4,0,1,0,2,0,0 & 1/252 \\
a &        & (4,9),2(3,7)$\succ$(10,23) ;(2,5),(3,8)$\succ$(5,13)&21/31583 \\
b &        & 2(2,5),(3,8)$\succ$(7,18) & 1/840 \\
c &        & (2,5),2(3,8)$\succ$(8,21) & 1/630 \\
97 & 4,0,0,1,1,1,1,1,1,1 & 7,0,1,3,0,2,2,0,6,0,0,1,0,0,1 & 43/13860 \\
a &        & (4,9),3(3,7)$\succ$(13,30) & 1/660 \\
b &        & (2,5),2(3,8)$\succ$(8,21) & 1/1386 \\
98 & 4,0,0,0,0,0,1,0,1,0 & 9,0,1,1,0,5,0,1,5,0,1,0,3,0,0 & 17/2772 \\
a &        & (4,9),(3,7)$\succ$(7,16);2(1,3),(4,11)$\succ$(6,17) & 3/1904 \\
b &        & 4(1,3),(4,11)$\succ$(8,23) & 5/5796 \\
99 & 4,0,0,0,1,0,1,1,0,0 & 10,0,1,2,0,2,0,1,7,0,1,0,1,1,0 & 41/13860 \\
a &        & (4,9),2(3,7)$\succ$(10,23) & 33/20914 \\
b &        & (1,4),(2,9)$\succ$(3,13) & 37/45045 \\
c &        & (1,3),(4,11)$\succ$(5,14) & 1/1260 \\
100 & 4,0,0,0,1,0,1,1,1,0 & 10,0,1,2,0,2,0,1,6,1,0,0,2,0,1 & 13/3465 \\
a &        & (4,11),2(3,7)$\succ$(10,23) & 3/1265 \\
b &        & (1,3),(3,10)$\succ$(4,13) & 27/22733 \\
c &        & (4,11),(1,3)$\succ$(5,14) & 1/630 \\
101 & 4,0,0,0,1,0,1,1,1,1 & 10,0,1,2,0,2,0,1,7,0,0,1,1,0,1 & 73/13860 \\
a &        & (3,11),(1,4)$\succ$(4,15);(4,9),2(3,7)$\succ$(10,23) & 3/1265 \\
b &        & (3,11),(1,4)$\succ$(4,15); (1,3),(4,11)$\succ$(5,14) & 1/630 \\
c &        & (4,9),2(3,7)$\succ$(10,23); (1,3),(4,11)$\succ$(5,14) & 29/16839 \\
d &        & 3(1,3),(4,11)$\succ$(7,20) & 1/1386 \\
e &        & (4,9),(3,7)$\succ$(7,16);(4,11),2(1,3)$\succ$(6,17)&19/26767\\
102 &4,0,0,1,1,1,1,1,1,1 & 7,0,1,3,0,2,1,1,5,0,1,0,1,0,1 & 71/27720 \\
a &        & (4,9),3(3,7)$\succ$(13,30) & 3/3080 \\
b &        & (3,8),(4,11)$\succ$(7,19);(4,9),3(3,7)$\succ$(13,30)  & 31/31920 \\
103 & 4,0,0,0,0,0,1,0,1,0 & 8,1,0,1,0,5,1,0,6,0,1,0,3,0,0 & 3/616 \\
a &        & 5(2,5),(3,8)$\succ$(13,33) & 1/924 \\
b &        & (1,2),(5,11)$\succ$(6,13) & 1/728 \\
104 & 4,0,0,0,1,0,1,1,0,0 & 9,1,0,2,0,2,1,0,8,0,1,0,1,1,0 & 47/27720 \\
105 & 4,0,0,1,0,1,0,1,1,0 & 4,1,0,3,0,6,0,0,6,0,1,0,2,0,0 & 1/770 \\
106 & 4,0,0,0,1,0,1,1,1,0 & 9,1,0,2,0,2,1,0,7,1,0,0,2,0,1 & 23/9240 \\
107 & 4,0,0,0,1,0,1,1,1,1 & 9,1,0,2,0,2,1,0,8,0,0,1,1,0,1 & 37/9240 \\
a &        & (3,11),(1,4)$\succ$(4,15) & 23/9240 \\
b &        & 2(2,5),(3,8)$\succ$(7,18) & 17/13860 \\
108 & 4,0,0,1,1,1,1,1,1,1 & 6,1,0,3,0,2,2,0,6,0,1,0,1,0,1 & 1/770 \\
109 & 4,0,0,0,1,0,1,1,1,1 & 9,1,0,2,0,2,0,1,7,0,1,0,2,0,1 & 4/1155 \\
a &        & (4,11),(1,3)$\succ$(5,14) & 1/770 \\
110 & 4,0,0,0,1,0,1,1,0,0 & 9,0,2,1,0,2,1,0,8,0,1,0,2,0,1 & 11/2520 \\
a &        & 2(4,9),(3,7)$\succ$(11,25);(2,5),(3,8)$\succ$(5,13) & 8/6825 \\
b &        & 2(2,5),(3,8)$\succ$(7,18) & 1/630 \\
111 & 4,0,0,0,0,0,1,1,1,0 & 8,0,2,0,0,6,0,0,7,0,0,1,2,0,0 & 7/990 \\
a &        & (3,11),2(1,4)$\succ$(5,19) & 4/855 \\
112 & 4,0,0,1,0,1,1,1,1,0 & 5,0,2,1,0,6,1,0,5,0,1,0,2,0,0 & 11/2520 \\
a &        & 2(4,9),(3,7)$\succ$(11,25);(2,5),(3,8)$\succ$(5,13) & 8/6825 \\
b &        & 4(2,5),(3,8)$\succ$(11,28) & 1/1260 \\
113 & 4,0,0,0,0,0,1,1,1,0 & 7,1,1,0,0,6,0,0,7,0,1,0,3,0,0 & 73/13860 \\
a &        & (5,11),(4,9)$\succ$(9,20) & 1/210 \\
b &        & (1,2),(5,11)$\succ$(6,13) & 1/945 \\
114 & 4,0,0,0,1,0,0,0,1,0 & 11,0,0,3,0,2,1,0,7,0,1,1,0,0,1 & 47/9240 \\
a &        & 2(2,5),(3,8)$\succ$(7,18);(2,7),(3,11)$\succ$(5,18) & 1/630 \\
115 & 4,0,0,0,1,0,0,0,1,0 & 11,0,0,3,0,2,0,1,6,0,2,0,1,0,1 & 1/220 \\
a &        & 2(1,3),(4,11)$\succ$(6,17) & 1/1020 \\
116 & 4,0,0,0,1,0,1,0,1,0 & 12,0,0,1,1,1,1,1,5,0,2,0,0,1,0 & 71/27720 \\
a &        & (3,7),(5,12)$\succ$(8,19);(3,8),(4,11)$\succ$(7,19) & 8/5985 \\
b &        & (5,12),(2,5)$\succ$(7,17);(3,8),(4,11)$\succ$(7,19) & 20/20349 \\
117 & 4,0,0,0,1,0,1,0,1,0 & 12,0,0,2,0,2,1,1,5,0,2,0,0,1,0 & 137/27720 \\
a &        & 2(2,5),(3,8)$\succ$(7,18) & 1/462 \\
b &        & (2,5),(3,8)$\succ$(5,13);(4,11),(1,3)$\succ$(5,14) & 1/1170 \\
c &        & (3,8),(4,11)$\succ$(7,19) & 26/5985 \\
d &        & (4,11),2(1,3)$\succ$(6,17) & 59/42840\\
118 & 4,0,0,0,1,0,1,0,1,0 & 11,0,1,0,1,1,2,0,6,0,2,0,1,0,1 & 1/252 \\
a &        & (5,12),(2,5)$\succ$(7,17) & 16/5355 \\
b &        & (2,5),2(3,8)$\succ$(8,21) & 1/630 \\
119 & 4,0,0,0,1,0,0,1,1,0 & 10,0,1,2,0,3,0,0,8,0,1,1,0,0,1 & 19/3465 \\
a &        & (4,9),2(3,7)$\succ$(10,23);(2,7),(3,11)$\succ$(5,18) & 7/2070 \\
120 & 4,0,0,0,1,0,1,0,1,0 & 11,0,1,1,0,2,2,0,6,0,2,0,1,0,1 & 2/315 \\
a &        & (4,9),(3,7)$\succ$(7,16);2(2,5),2(3,8)$\succ$2(5,13) & 11/7280 \\
121 & 4,0,0,0,1,1,1,0,1,0 & 9,0,1,2,0,2,2,0,4,1,1,0,2,0,0 & 1/252 \\
a &        & (4,9),2(3,7)$\succ$(10,23);(3,10),(2,7)$\succ$(5,17) & 41/23460 \\
b &        & (4,9),2(3,7)$\succ$(10,23);(2,5),(3,8)$\succ$(5,13) & 21/31583 \\
c &        & (4,9),(3,7)$\succ$(7,16);(3,10),(2,7)$\succ$(5,17) &61/28560\\
d &        & 2(2,5),(3,8)$\succ$(7,18) & 1/840 \\
e &        & (2,5),2(3,8)$\succ$(8,21);(3,10),(2,7)$\succ$(5,17) & 4/5355 \\
f &        & (1,3),(3,10)$\succ$(4,13) & 23/16380 \\
122 & 4,0,0,0,1,1,1,0,1,1 & 9,0,1,2,0,2,2,0,5,0,1,1,1,0,0 & 19/3465 \\
a &        & (4,9),2(3,7)$\succ$(10,23);(2,7),(3,11)$\succ$(5,18);(2,5),2(3,8)$\succ$(8,21) & 29/28980 \\
b &        & (4,9),2(3,7)$\succ$(10,23);(3,11),(1,4)$\succ$(4,15);(2,5),(3,8)$\succ$(5,13) & 21/31583 \\
c&&(4,9),(3,7)$\succ$(7,16);,2(2,5),2(3,8)$\succ$2(5,13)&31/48048\\
d &        & (4,9),(3,7)$\succ$(7,16);(2,7),(3,11)$\succ$(5,18),2(2,5),(3,8)$\succ$(7,18) & 1/1008 \\
e&        & (4,9),(3,7)$\succ$(7,16); (3,11),(1,4)$\succ$(4,15); (2,5),(3,8)$\succ$(5,13) & 23/21840 \\
e &        & (2,7),(3,11)$\succ$(5,18);2(2,5),2(3,8)$\succ$2(5,13) & 1/1092 \\
123 & 4,0,0,0,1,0,1,1,1,0 & 11,0,1,0,1,2,0,1,6,0,2,0,0,1,0 & 41/13860 \\
a &        & (5,12),2(2,5)$\succ$(9,22) & 1/693 \\
b &        & (4,11),(1,3)$\succ$(5,14) & 1/1260 \\
124 & 4,0,0,0,1,0,1,1,1,0 & 11,0,1,1,0,3,0,1,6,0,2,0,0,1,0 & 37/6930 \\
a &        & (4,9),(3,7)$\succ$(7,16);2(1,3),(4,11)$\succ$(6,17) & 31/39643 \\
b &        & 3(1,3),(4,11)$\succ$(7,20) & 31/39643 \\
125 & 4,0,0,0,1,1,1,0,1,1 & 9,0,1,2,0,2,1,1,4,0,2,0,2,0,0 & 137/27720 \\
a &        & (4,9),2(3,7)$\succ$(10,23);2(2,5),(3,8)$\succ$(7,18) & 25/31878 \\
b &        & (4,9),2(3,7)$\succ$(10,23);(4,11),(1,3)$\succ$(5,14) & 9/6440 \\
c &        & (4,9),2(3,7)$\succ$(10,23);(3,8),(4,11)$\succ$(7,19) & 59/19906 \\
d &        & (4,11),2(1,3)$\succ$(6,17) & 59/42840 \\
e &        & 2(2,5),(3,8)$\succ$(7,18);(4,11),(1,3)$\succ$(5,14) & 1/1170 \\
f &        & (4,9),(3,7)$\succ$(7,16);(4,11),(1,3)$\succ$(5,14) &1/560\\
126 & 4,0,0,0,1,0,0,1,1,0 & 9,1,0,2,0,3,0,0,8,0,2,0,1,0,1 & 17/4620 \\
127 & 4,0,0,0,1,0,1,1,1,0 & 10,1,0,0,1,2,1,0,7,0,2,0,0,1,0 & 47/27720 \\
a &        & (5,12),(2,5)$\succ$(7,17) & 68/95087 \\
128 & 4,0,0,0,1,0,1,1,1,0 & 10,1,0,1,0,3,1,0,7,0,2,0,0,1,0 & 42/10303 \\
a &        & 3(2,5),(3,8)$\succ$(9,23) & 13/15939 \\
129 & 4,0,0,0,1,1,1,0,1,1 & 8,1,0,2,0,2,2,0,5,0,2,0,2,0,0 & 17/4620 \\
a &        & 2(2,5),(3,8)$\succ$(7,18) & 5/5544 \\
b &        & (2,5),2(3,8)$\succ$(8,21) & 1/770 \\
130 & 4,0,0,0,1,0,1,1,1,0 & 10,0,2,0,0,3,1,0,7,0,2,0,1,0,1 & 17/2520 \\
a &        & 3(2,5),(3,8)$\succ$(9,23) & 101/28980 \\
131 & 4,0,0,0,1,1,1,1,1,0 & 8,0,2,1,0,3,1,0,5,1,1,0,2,0,0 & 11/2520 \\
a &        & 2(4,9),(3,7)$\succ$(11,25);(2,5),(3,8)$\succ$(5,13) & 13/4200 \\
b &        & 2(4,9),(3,7)$\succ$(11,25);(3,10),(2,7)$\succ$(5,17) & 23/10200 \\
c &        & (4,9),(3,7)$\succ$(7,16);(1,3),(3,10)$\succ$(4,13) & 31/38323 \\
d &        & 3(2,5),(3,8)$\succ$(9,23) & 8/7245 \\
e &        & 2(2,5),(3,8)$\succ$(7,18),(3,10),(2,7)$\succ$(5,17) & 4/5355 \\
f &        & (1,3),(3,10)$\succ$(4,13) & 59/32760 \\
132 & 4,0,0,0,1,1,1,1,1,1 & 8,0,2,1,0,3,1,0,6,0,1,1,1,0,0 & 49/8333 \\
a &        & 2(4,9),(3,7)$\succ$(11,25);(2,7),(3,11)$\succ$(5,18);2(2,5),(3,8)$\succ$(7,18) & 1/900 \\
b &        & 2(4,9),(3,7)$\succ$(11,25);3(2,5),(3,8)$\succ$(9,23) & 239/177100 \\
c &        & (4,9),(3,7)$\succ$(7,16);(2,7),(3,11)$\succ$(5,18);3(2,5),(3,8)$\succ$(9,23) & 1/1104 \\
d &        & 2(4,9),(3,7)$\succ$(11,25);(3,11),(1,4)$\succ$(4,15);(2,5),(3,8)$\succ$(5,13) & 8/6825 \\
e &        & (3,11),(1,4)$\succ$(4,15); 3(2,5),(3,8)$\succ$(9,23) & 8/7245 \\

133 & 4,0,0,0,1,1,1,1,1,1 & 8,0,2,1,0,3,0,1,5,0,2,0,2,0,0 & 37/6930 \\
a &        & 2(4,9),(3,7)$\succ$(11,25);(4,11),(1,3)$\succ$(5,14) & 1/525 \\
b &        & (4,9),(3,7)$\succ$(7,16); (4,11),2(1,3)$\succ$(6,17) & 31/39643 \\
c &        & (4,11),3(1,3)$\succ$(7,20) & 1/1260 \\
134 & 4,0,0,0,1,1,1,1,1,1 & 7,1,1,1,0,3,1,0,6,0,2,0,2,0,0 & 42/10303 \\
a &        & (5,11),(4,9)$\succ$(9,20);2(2,5),(3,8)$\succ$(7,18) & 1/1260 \\
b &        & (4,9),(3,7)$\succ$(7,16); (2,5),(3,8)$\succ$(5,13) & 40/34443 \\
c &        & 3(2,5),(3,8)$\succ$(9,23) & 13/15939 \\
135 & 4,0,0,0,1,1,0,0,1,0 & 9,0,1,2,0,3,1,0,5,0,3,0,1,0,0 & 13/2520 \\
a &        & (4,9),2(3,7)$\succ$(10,23);2(2,5),(3,8)$\succ$(7,18) & 29/28980 \\
b &        & (4,9),(3,7)$\succ$(7,16);3(2,5),(3,8)$\succ$(9,23) & 1/1104 \\
136 & 4,0,0,0,1,1,0,1,1,0 & 8,0,2,1,0,4,0,0,6,0,3,0,1,0,0 & 1/180 \\
a &        & 2(4,9),(3,7)$\succ$(11,25) & 3/700 \\
137 & 5,0,0,0,1,0,1,0,1,0 & 14,0,0,3,0,3,1,1,7,0,1,1,0,1,0 & 42/10303 \\
a &        & (2,7),(3,11)$\succ$(5,18);(2,5),(3,8)$\succ$(5,13) & 43/30030 \\
b &        & (2,7),(3,11)$\succ$(5,18);(4,11),(1,3)$\succ$(5,14) & 1/840 \\
c &        & (2,7),(3,11)$\succ$(5,18);(3,8),(4,11)$\succ$(7,19) & 11/3990 \\
d &        & 3(2,5),(3,8)$\succ$(9,23) & 13/15939 \\
e &        & (4,11),(1,3)$\succ$(5,14) & 53/27720 \\
138 & 5,0,0,0,1,0,1,0,1,0 & 14,0,0,3,0,3,0,2,6,0,2,0,1,1,0 & 7/1980 \\
a &        & (1,4),(2,9)$\succ$(3,13) & 1/715 \\
b &        & 2(4,11),(1,3)$\succ$(9,25) & 1/900 \\
139 & 5,0,0,0,1,0,1,0,1,0 & 13,0,1,2,0,3,2,0,8,0,1,1,1,0,1 & 19/3465 \\
a &        & (4,9),2(3,7)$\succ$(10,23);(2,7),(3,11)$\succ$(5,18);(2,5),2(3,8)$\succ$(8,21) & 29/28980 \\
b &        & (4,9),2(3,7)$\succ$(10,23);(3,11),(1,4)$\succ$(4,15);(2,5),(3,8)$\succ$(5£¬13)& 21/31583  \\
c &        & (4,9),(3,7)$\succ$(7,16);(2,7),(3,11)$\succ$(5,18);2(2,5),(3,8)$\succ$(7,18) & 1/1008 \\
d &        & (3,11),(1,4)$\succ$(4,15); (2,5),2(3,8)$\succ$(8,21) & 1/630 \\
e &        & (3,11),(1,4)$\succ$(4,15);3(2,5),(3,8)$\succ$(7,18) & 41/57960 \\
f &        & (2,7),(3,11)$\succ$(5,18);3(2,5),(3,8)$\succ$(9,23) & 29/19320 \\
g &        & (2,7),(3,11)$\succ$(5,18);2(2,5),2(3,8)$\succ$2(5,13) & 1/1092 \\
140 & 5,0,0,0,1,0,1,0,1,0 & 13,0,1,2,0,3,1,1,7,0,2,0,2,0,1 & 137/27720 \\
a &        & (4,9),2(3,7)$\succ$(10,23);2(2,5),(3,8)$\succ$(7,18) & 25/31878 \\
b &        & (4,9),2(3,7)$\succ$(10,23);(4,11),(1,3)$\succ$(5,14) & 9/6440 \\
c &        & (4,9),2(3,7)$\succ$(10,23);(3,8),(4,11)$\succ$(7,19) & 59/19906 \\
d &        & (4,9),(3,7)$\succ$(7,16);3(2,5),(3,8)$\succ$(9,23) & 17/24661\\
e &        & (2,5),(3,8)$\succ$(5,13);(4,11),(1,3)$\succ$(5,14) & 1/1170 \\
f &        & 2(1,3),(4,11)$\succ$(6,17) & 59/42840 \\
141 & 5,0,0,0,1,0,1,1,1,0 & 13,0,1,2,0,4,0,1,8,0,1,1,0,1,0 & 31/6930 \\
a &        & (4,9),2(3,7)$\succ$(10,23);(2,7),(3,11)$\succ$(5,18) & 3/1808 \\
b &        & (4,9),2(3,7)$\succ$(10,23);(1,3),(4,11)$\succ$(5,14) & 25/26924 \\
c &        & (2,7),(3,11)$\succ$(5,18);(1,3),(4,11)$\succ$(5,14) & 1/630\\
142 & 5,0,0,0,1,0,1,0,1,0 & 12,1,0,2,0,3,2,0,8,0,2,0,2,0,1 & 17/4620 \\
a &        & 2(2,5),(3,8)$\succ$(7,18) & 5/5544 \\
b &        & (2,5),2(3,8)$\succ$(8,21) & 1/770 \\
143 & 5,0,0,0,1,0,1,1,1,0 & 12,1,0,2,0,4,1,0,9,0,1,1,0,1,0 & 89/27720 \\
a &        & (2,7),(3,11)$\succ$(5,18) & 23/9240 \\
b &        & (2,5),(3,8)$\succ$(5,13) & 58/45045 \\
144 & 5,0,0,0,1,0,1,1,1,0 & 12,1,0,2,0,4,0,1,8,0,2,0,1,1,0 & 37/13860 \\
145 & 5,0,0,0,1,0,1,1,1,0 & 12,0,2,1,0,4,1,0,9,0,1,1,1,0,1 & 49/8333 \\
a &        & (2,7),(3,11)$\succ$(5,18);2(4,9),(3,7)$\succ$(11,25);2(2,5),(3,8)$\succ$(7,18) & 1/900 \\
b &        & (2,7),(3,11)$\succ$(5,18);(4,9),(3,7)$\succ$(7,16);3(2,5),(3,8)$\succ$(9,23) & 1/1104 \\
c &        & (2,7),(3,11)$\succ$(5,18); 4(2,5),(3,8)$\succ$(11,28) & 1/630 \\
d &        & (3,11),(1,4)$\succ$(4,15);2(4,9),(3,7)$\succ$(11,25);(2,5),(3,8)$\succ$(5,13) & 8/6825 \\
e &        & (3,11),(1,4)$\succ$(4,15); 4(2,5),(3,8)$\succ$(11,28) & 1/1260 \\
f &        & (3,11),(1,4)$\succ$(4,15);(4,9),(3,7)$\succ$(7,16);(2,5),(3,8)$\succ$(5,13)&19/13104\\
g &        & 2(4,9),(3,7)$\succ$(11,25); 4(2,5),(3,8)$\succ$(11,28) & 2/1925 \\
146 & 5,0,0,0,1,0,1,1,1,0 & 12,0,2,1,0,4,0,1,8,0,2,0,2,0,1 & 37/6930 \\
a &        & 2(4,9),(3,7)$\succ$(11,25);(4,11),(1,3)$\succ$(5,14) & 1/525 \\
b &        & (4,9),(3,7)$\succ$(7,16);(4,11),2(1,3)$\succ$(6,17) & 31/39643 \\
c &        & (4,11),3(1,3)$\succ$(7,20) & 1/1260 \\
147 & 5,0,0,0,1,0,1,1,1,0 & 11,1,1,1,0,4,1,0,9,0,2,0,2,0,1 & 42/10303 \\
a &        & (5,11),(4,9)$\succ$(9,20);2(2,5),(3,8)$\succ$(7,18) & 1/1260 \\
b &        & (4,9),(3,7)$\succ$(7,16);(2,5),(3,8)$\succ$(5,13) & 19/6160 \\
c &        & 3(2,5),(3,8)$\succ$(9,23) & 13/15939 \\
148 & 5,0,0,0,1,1,1,0,1,0 & 11,0,2,1,0,4,2,0,6,0,3,0,2,0,0 & 1/180 \\
a &        & 2(4,9),(3,7)$\succ$(11,25);4(2,5),(3,8)$\succ$(11,28) & 1/1400 \\
b &        & 2(4,9),(3,7)$\succ$(11,25);(2,5),2(3,8)$\succ$(8,21) & 1/525 \\
c &        & 3(2,5),2(3,8)$\succ$(12,31) & 1/1395 \\
149 & 5,0,0,0,1,1,1,1,1,0 & 10,0,3,0,0,5,1,0,7,0,3,0,2,0,0 & 1/168 \\
a &        & 5(2,5),(3,8)$\succ$(13,33) & 1/462 \\
\noalign{\vskip .4ex \hrule height 0.08em\vskip.65ex}
\end{longtable}\end{center}

\subsection{Classification of Case 2}
Compared to Case (i), the difference is that we have 
 $\epsilon=2$ and $\sigma_5=1$.  So we may do parallel calculations. Since the expressions for $B^{(n)}$ are different from before, we should still list them here in order to justify our result. 
 
 \begin{equation*}
B^{(0)}
\begin{cases}
n^0_{1,2}=5{\chi}-4P_3+P_4\\
n^0_{1,3}=4{\chi}+2P_3-3P_4+P_5\\
n^0_{1,4}={\chi}+P_3+2P_4-P_5-1\\
n^0_{1,5}=0\\
n^0_{1,6}=1
\end{cases}
\end{equation*}
$$
\epsilon_5=\Delta^5(B^{(0)})-\Delta^5(B)=2{\chi}-P_3+2P_5-P_6-1
$$
\begin{equation*}
B^{(5)}
\begin{cases}
n^5_{1,2}=3{\chi}-3P_3+P_4-2P_5+P_6+1\\
n^5_{2,5}=2{\chi}-P_3+2P_5-P_6-1\\
n^5_{1,3}=2{\chi}+3P_3-3P_4-P_5+P_6+1\\
n^5_{1,4}={\chi}+P_3+2P_4-P_5-1\\
n^5_{1,5}=0\\
n^5_{1,6}=1
\end{cases}
\end{equation*}
$$
\epsilon_7={\chi}-P_3+P_6+P_7-P_8-1
$$
\begin{equation*}
B^{(7)}
\begin{cases}
n^7_{1,2}=2{\chi}-2P_3+P_4-2P_5-P_7+P_8+2+\eta\\
n^7_{3,7}={\chi}-P_3+P_6+P_7-P_8-1-\eta\\
n^7_{2,5}={\chi}+2P_5-2P_6-P_7+P_8+\eta\\
n^7_{1,3}=2{\chi}+3P_3-3P_4-P_5+P_6+1-\eta\\
n^7_{2,7}=\eta\\
n^7_{1,4}={\chi}+P_3+2P_4-P_5-1-\eta\\
n^7_{1,5}=0\\
n^7_{1,6}=1
\end{cases}
\end{equation*}
$$
\epsilon_8=-P_3-P_4+P_5+P_6+P_8-P_9-1
$$
\begin{equation*}
B^{(8)}
\begin{cases}
n^8_{1,2}=2{\chi}-2P_3+P_4-2P_5-P_7+P_8+2+\eta\\
n^8_{3,7}={\chi}-P_3+P_6+P_7-P_8-1-\eta\\
n^8_{2,5}={\chi}+P_3+P_4+P_5-3P_6-P_7+P_9+1+\eta\\
n^8_{3,8}=-P_3-P_4+P_5+P_6+P_8-P_9-1\\
n^8_{1,3}=2{\chi}+4P_3-2P_4-2P_5-P_8+P_9+2-\eta\\
n^8_{2,7}=\eta\\
n^8_{1,4}={\chi}+P_3+2P_4-P_5-1-\eta\\
n^8_{1,5}=0\\
n^8_{1,6}=1
\end{cases}
\end{equation*}
$$
\epsilon_9=-2P_3+P_4+P_5-P_7+P_8+P_9-P_{10}-1+\eta
$$
\begin{equation*}
B^{(9)}
\begin{cases}
n^9_{1,2}=2{\chi}-3P_5-P_9+P_{10}+3\\
n^9_{4,9}=-2P_3+P_4+P_5-P_7+P_8+P_9-P_{10}-1+\eta\\
n^9_{3,7}={\chi}+P_3-P_4-P_5+P_6+2P_7-2P_8-P_9+P_{10}-2\eta\\
n^9_{2,5}={\chi}+P_3+P_4+P_5-3P_6-P_7+P_9+1+\eta\\
n^9_{3,8}=-P_3-P_4+P_5+P_6+P_8-P_9-1\\
n^9_{1,3}=2{\chi}+4P_3-2P_4-2P_5-P_8+P_9+2-\eta\\
n^9_{2,7}=\eta\\
n^9_{1,4}={\chi}+P_3+2P_4-P_5-1-\eta\\
n^9_{1,5}=0\\
n^9_{1,6}=1
\end{cases}
\end{equation*}
$$
\epsilon_{10}=-P_3+2P_6+P_{10}-P_{11}-2-\eta
$$
\begin{equation*}
B^{(10)}
\begin{cases}
n^{10}_{1,2}=2{\chi}-3P_5-P_9+P_{10}+3\\
n^{10}_{4,9}=-2P_3+P_4+P_5-P_7+P_8+P_9-P_{10}-1+\eta\\
n^{10}_{3,7}={\chi}+P_3-P_4-P_5+P_6+2P_7-2P_8-P_9+P_{10}-2\eta\\
n^{10}_{2,5}={\chi}+P_3+P_4+P_5-3P_6-P_7+P_9+1+\eta\\
n^{10}_{3,8}=-P_3-P_4+P_5+P_6+P_8-P_9-1\\
n^{10}_{1,3}=2{\chi}+5P_3-2P_4-2P_5-2P_6-P_8+P_9-P_{10}+P_{11}+4\\
n^{10}_{3,10}=-P_3+2P_6+P_{10}-P_{11}-2-\eta\\
n^{10}_{2,7}=P_3-2P_6-P_{10}+P_{11}+2+2\eta\\
n^{10}_{1,4}={\chi}+P_3+2P_4-P_5-1-\eta\\
n^{10}_{1,5}=0\\
n^{10}_{1,6}=1
\end{cases}
\end{equation*}
By computing $\Delta^{11}(B^{(10)})$, We get
$$
\epsilon_{11}={\chi}+P_3-P_5-P_8+P_{10}+P_{11}-2-\eta
$$

\begin{equation*}
B^{(11)}
\begin{cases}
n^{11}_{1,2}=2{\chi}-3P_5-P_9+P_{10}+3-\alpha\\
n^{11}_{5,11}=\alpha\\
n^{11}_{4,9}=-2P_3+P_4+P_5-P_7+P_8+P_9-P_{10}-1+\eta-\alpha\\
n^{11}_{3,7}={\chi}+P_3-P_4-P_5+P_6+2P_7-2P_8-P_9+P_{10}-2\eta\\
n^{11}_{2,5}={\chi}+P_3+P_4+P_5-3P_6-P_7+P_9+1+\eta\\
n^{11}_{3,8}=-P_3-P_4+P_5+P_6+P_8-P_9-1-\beta\\
n^{11}_{4,11}=\beta\\
n^{11}_{1,3}=2{\chi}+5P_3-2P_4-2P_5-2P_6-P_8+P_9-P_{10}+P_{11}+4-\beta\\
n^{11}_{3,10}=-P_3+2P_6+P_{10}-P_{11}-2-\eta\\
n^{11}_{2,7}=-{\chi}+P_5-2P_6+P_8-2P_{10}+4+3\eta+\alpha+\beta\\
n^{11}_{3,11}={\chi}+P_3-P_5-P_8+P_{10}+P_{11}-2-\eta-\alpha-\beta\\
n^{11}_{1,4}=2P_4+P_8-P_{10}-P_{11}+1+\alpha+\beta\\
n^{11}_{1,5}=0\\
n^{11}_{1,6}=1
\end{cases}
\end{equation*}
$$
\epsilon_{12}=-{\chi}-3P_3+2P_5+P_6-P_7+P_8-P_{13}-1-\eta
$$
\begin{equation*}
B^{(12)}
\begin{cases}
n^{12}_{1,2}=2{\chi}-3P_5-P_9+P_{10}+3-\alpha\\
n^{12}_{5,11}=\alpha\\
n^{12}_{4,9}=-2P_3+P_4+P_5-P_7+P_8+P_9-P_{10}-1+\eta-\alpha\\
n^{12}_{3,7}=2{\chi}+4P_3-P_4-3P_5+3P_7-3P_8-P_9+P_{10}+P_{13}+1-\eta\\
n^{12}_{5,12}=-{\chi}-3P_3+2P_5+P_6-P_7+P_8-P_{13}-1-\eta\\
n^{12}_{2,5}=2{\chi}+4P_3+P_4-P_5-4P_6-P_8+P_9+P_{13}+2+2\eta\\
n^{12}_{3,8}=-P_3-P_4+P_5+P_6+P_8-P_9-1-\beta\\
n^{12}_{4,11}=\beta\\
n^{12}_{1,3}=2{\chi}+5P_3-2P_4-2P_5-2P_6-P_8+P_9-P_{10}+P_{11}+4-\beta\\
n^{12}_{3,10}=-P_3+2P_6+P_{10}-P_{11}-2-\eta\\
n^{12}_{2,7}=-{\chi}+P_5-2P_6+P_8-2P_{10}+4+3\eta+\alpha+\beta\\
n^{12}_{3,11}={\chi}+P_3-P_5-P_8+P_{10}+P_{11}-2-\eta-\alpha-\beta\\
n^{12}_{1,4}=2P_4+P_8-P_{10}-P_{11}+1+\alpha+\beta\\
n^{12}_{1,5}=0\\
n^{12}_{1,6}=1
\end{cases}
\end{equation*}

Similarly we have the boundedness for $\chi$, $\eta$, $\alpha$, $\beta$ and $P_{13}$, therefore the boundedness of $B^{(12)}$. Here we have:
 \begin{equation}
{\chi}\leq 3.
\end{equation}
\begin{equation}
\eta\leq{\chi}+2.
\end{equation}
\begin{equation}
\alpha\leq 3+\eta.
\end{equation}
\begin{equation}
\beta\leq 2
\end{equation}
Finally, we have
\begin{equation}
P_{13}\leq-{\chi}+4.
\end{equation}

The result after running our computer program is unexpectedly simple. 
\par\nobreak
\LTleft=-10pt \LTright=0pt
\begin{center}
\scriptsize
\begin{longtable}{p{4mm}p{25mm}p{80mm}p{10mm}}
\noalign{\hrule height .08em\vskip.8ex}
No & $({\chi},P_3,P_4,\ldots,P_{11})$ & $B^{(12)}=(n_{1,2}^{12},n_{5,11}^{12},\ldots,n_{1,4}^{12},n_{1,6}^{12})$ & $K^3$\quad \\
\noalign{\vskip .4ex \hrule height 0.05em\vskip.65ex}
1 & 2,0,0,1,1,0,1,0,1,1 &5,0,0,1,0,1,2,0,3,0,0,0,0,1 &1/420 \\
2 & 2,0,0,1,1,0,1,1,1,1 &4,0,1,0,0,2,1,0,4,0,0,0,0,1 &1/360 \\
a & & (2,5),(3,8)$\succ$(5,13) & 1/1170 \\
\noalign{\vskip .4ex \hrule height 0.08em\vskip.65ex}
\end{longtable}\end{center}

\begin{thm}
Let X be a nonsingular projective 3-fold of general type with the pluricanonical section index  $\delta(X)=12$. Then $\Vol(X)\geq\frac{31}{48048}$ and the equality holds if and only it corresponds to Case 122.c.
\end{thm}
\begin{proof} When $P_{12}\geq3$, one has $Vol(X)\geq\frac{1}{1224}$ by those inequalities in Chen--Chen \cite[Section 3]{Chen3}.

 If $P_{12}=2$, we can get that $Vol(X)\geq\frac{31}{48048}$ from our tables and the equality holds if and only if it corresponds to Case 122.c. 
\end{proof}

\begin{rem} It is mysterious whether the lower bound $\frac{31}{48048}$ is optimal or not.
\end{rem}


\end{document}